\documentclass{amsart}

\usepackage{lineno}
\usepackage{amssymb}
\usepackage{amsmath}
\usepackage{amsthm}
\usepackage{mathrsfs}
\usepackage{bbm}
\usepackage{cite}
\usepackage{color}
\usepackage{enumitem}
\usepackage{pgfplots}

\usepackage[colorlinks,hypertexnames=false]{hyperref}
\usepackage[msc-links]{amsrefs}

\usepackage{changes}

\usepackage{mathtools}
\usepackage{extarrows}
\usepackage{setspace}

\usepackage{graphicx}

\usepackage{imakeidx}
\makeindex

\hypersetup{
	colorlinks=true,
	linkcolor=black,
	filecolor=black,
	urlcolor=black,
	citecolor=black,
}

\newtheorem{thm}{Theorem}[section]

\newtheorem{defn}[thm]{Definition}
\newtheorem{lem}[thm]{Lemma}
\newtheorem{prop}[thm]{Proposition}

\newtheorem{corollary}[thm]{Corollary}

\newtheorem{remark}[thm]{Remark}
\newtheorem{example}[thm]{Example}

\numberwithin{equation}{section}

\pgfplotsset{compat=1.18} 

\begin{document}

\title[SFR functions on arbitrary domains]{Two-sided slice regular functions}
\author{Xinyuan Dou}
\email[Xinyuan Dou]{douxinyuan@ustc.edu.cn}
\address{Department of Mathematics, University of Science and Technology of China, Hefei 230026, China}
\author{Haosen Liu}
\email[Haosen Liu]{2949380535@qq.com}
\address{Department of Mathematics, University of Science and Technology of China, Hefei 230026, China}  
\author{Ting Yang*}
\email[Ting Yang (The corresponding author)]{tingy@mail.ustc.edu.cn}
\address{School of Mathematics and Statistics, Anqing Normal University, Anqing, 246133, China}

%\date{\today}
\thanks{This work was supported by Xiaomi Young Talents Program and by the NNSF of China (12401104). }


\begin{abstract}

In this paper, we introduce two types of two-sided slice regular functions, which are left slice regular in the first variable and right slice regular in the second variable. We prove that these two classes of functions coincide on axially symmetric domains.
We further investigate fundamental properties of two-sided slice regular functions and define a star-product that preserves this regularity. In addition, we establish the Cauchy integral formula for these functions. 
\end{abstract}

\maketitle
\section{Introduction}
The theory of slice regular functions, originally referred to as Cullen-regular functions, was introduced in 2006 by G. Gentili and D. C. Struppa \cites{GentiliStruppa2006,gentili2007new4}. Since then, slice regularity has been extended to a variety of algebraic settings, including Clifford algebras, octonions, real alternative algebras, and several-variable frameworks \cites{Colombo2009SliceMonogenic,gentili2010regular5,colombo2007functional,ghiloni2011slice,GhiloniPerotti2012Clifford,ColomboSabadiniStruppa2012,RenYang2020,GhiloniPerotti2022}. While the classical theory was primarily developed on axially symmetric slice domains, more recent studies have investigated slice regularity on more general classes of domains \cites{Dou202000426,GhiloniStoppato2024,Perotti2024}.

The theory of slice regular functions provides a natural framework for extending complex analysis to the quaternionic setting. Its development has also led to a corresponding spectral theory for quaternionic linear operators. In contrast to the complex case, the non-commutativity of quaternionic multiplication prevents a direct extension of the classical resolvent and spectrum. To overcome this difficulty, Colombo and Sabadini introduced the notion of the S-spectrum  \cite{colombo2007functional}. For a bounded right-linear operator $T$ on a two-sided quaternionic Banach space, the $S$-spectrum is defined by
$$
\sigma_S(T)
=
\left\{
s\in\mathbb H:
T^2-2\operatorname{Re}(s)T+|s|^2\mathcal I
\ \text{is not invertible}
\right\}.
$$
The associated left and right $S$-resolvent operators are given, for $s\notin\sigma_S(T)$, by
$$
S_L^{-1}(s,T)
=
-\left(
T^2-2\operatorname{Re}(s)T+|s|^2\mathcal I
\right)^{-1}(T-\overline{s}\mathcal I)
$$
and
$$
S_R^{-1}(s,T)
=
-(T-\overline{s}\mathcal I)
\left(
T^2-2\operatorname{Re}(s)T+|s|^2\mathcal I
\right)^{-1}.
$$
These operators provide the basic kernels of the $S$-functional calculus and play the role of the classical resolvent in quaternionic spectral theory \cites{ColomboSabadini2009,ColomboSabadini2010lr,colombo2018spectral,ColomboGentiliSabadiniStruppa2010}.

At the scalar level, the corresponding objects are the non-commutative Cauchy kernels, also known as slice hyperholomorphic Cauchy kernels. They may be regarded as the scalar counterparts of the $S$-resolvent operators and provide the analytic kernels underlying the $S$-functional calculus. More precisely, for $s,q\in\mathbb H$ with $q\notin[s]$, the left and right Cauchy kernels are defined, respectively, by
\begin{equation}\label{S-LR}
S_L^{-1}(s,q)=
-\left(q^2-2\operatorname{Re}(s)q+|s|^2\right)^{-1}(q-\overline{s}),
\end{equation}
and
\begin{equation}\label{S-RR}
S_R^{-1}(s,q)=
-(q-\overline{s})
\left(q^2-2\operatorname{Re}(s)q+|s|^2\right)^{-1}.
\end{equation}
These kernels yield Cauchy-type integral representations for slice regular functions and, upon replacing the quaternionic variable $q$ by an operator $T$, give rise to the corresponding $S$-resolvent operators. Thus, they provide a natural bridge between slice regular function theory and quaternionic operator theory.

An important feature of these kernels is their intrinsic two-sided structure. More precisely,
\begin{enumerate}
\item[(i)] $S_L^{-1}(s,q)$ is right slice regular with respect to the first variable and left slice regular with respect to the second variable;
\item[(ii)] $S_R^{-1}(s,q)$ is left slice regular with respect to the first variable and right slice regular with respect to the second variable.
\end{enumerate}
Moreover, they satisfy the symmetry relation
$$
S_L^{-1}(s,q)=-S_R^{-1}(q,s).
$$
This two-sided structure naturally suggests the study of functions of two quaternionic variables in which left and right slice regularity are considered simultaneously.

In this paper, we study functions of two quaternionic variables that are left slice regular in the first variable and right slice regular in the second variable. We characterize this class of functions in two ways: first, in terms of the holomorphicity of their restrictions to complex slices, and second, in terms of associated holomorphic stem functions. We establish several analytic properties of these functions and prove the equivalence of these two characterizations on axially symmetric domains. Furthermore, we introduce a corresponding star-product for this class of functions and establish a Cauchy integral formula for two-sided slice regular functions.

Reference \cite{GhiloniPerotti2022} develops a theory of slice regular functions of several variables over real alternative \(*\)-algebras based on stem and slice functions, where regularity is characterized indirectly through stem functions. In contrast, we mainly study weak slice regular functions, whose restrictions to complex slices are holomorphic. Since the theory of weak slice regular functions necessarily requires the framework of slice-cones \cite{Dou202000425}, in which additional technical difficulties arise, we only undertake a preliminary investigation of the simpler cases in this paper.
\section{Two-sided slice regular functions}
In this section, we introduce the notion of two-sided slice regular functions on axially symmetric domains and establish several fundamental results, including a uniqueness theorem and a representation formula. The latter shows that a two-sided slice regular function can be recovered from its values on an arbitrary full slice.

\begin{defn}
    The sphere $\mathbb S$ of imaginary units of quaternions $\mathbb H$ is defined by
    $$\mathbb S = \{q \in \mathbb  H : q^2 = -1\}.$$
\end{defn}

\begin{defn}
    A subset $\Omega $ of $\mathbb{H}$ is called slice-open, if $\Omega_I:=\Omega \cap \mathbb{C}_I$ is open in $\mathbb{C}_I$  for any $I\in \mathbb{S}$.
\end{defn}

\begin{defn}
    The slice topology is defined by
    $$ \tau_s(\mathbb{H}):=\{ \Omega\subset \mathbb{H}:\Omega\ \text{is slice-open} \}.$$
\end{defn}

\begin{defn}\label{weak regular}
    Let $\Omega_1,\Omega_2\in \tau_s(\mathbb{H}),$ let $f$ be a quaternion-valued function defined on $\Omega_1\times\Omega_2$. For each $I,J\in \mathbb{S},$ let $\Omega_{1,I}:=\Omega_1\cap \mathbb{C}_I,$ and $ \Omega_{2,J}:=\Omega_2\cap \mathbb{C}_J$ and $f|_{I,J}=f|_{\Omega_{1,I}\times\Omega_{2,J}}.$ If $f|_{I,J}$ has continuous partial derivatives and satisfies
   \begin{equation} \label{tsr}
   \left\{\begin{array}{cc}
      \frac{1}{2}\left(\frac{\partial}{\partial x}
  f(p,q)+I\frac{\partial}{\partial y}f(p,q)\right)=0\\ \\
       \frac{1}{2}\left(\frac{\partial}{\partial u}f(p,q)
  +\frac{\partial}{\partial v}f(p,q)J\right)=0
    \end{array}\right.
\end{equation}
for any $(p,q)=(x+Iy,u+Jv)\in \Omega_1\times \Omega_2$ with  $I,J\in \mathbb{S},$ then the function $f$ is called a $($two-sided$)$ slice regular function on $\Omega_1\times \Omega_2.$
Denote the class of two-sided slice regular functions on $\Omega_1\times \Omega_2$ by $\mathcal{SR}(\Omega_1\times \Omega_2).$
\end{defn}

\begin{remark}
The formula (\ref{tsr}) is also denoted as 
\begin{equation*}
   \left\{\begin{array}{cc}
      \frac{1}{2}\left(\frac{\partial}{\partial x}
    +I\frac{\partial}{\partial y}\right)f(p,q)=0\\ \\
       \frac{1}{2}\left(\frac{\partial}{\partial u}
  +\frac{\partial}{\partial v}R_J\right)f(p,q)=0,
    \end{array}\right.
\end{equation*}
where $R_J:\mathbb H \rightarrow \mathbb H$ denotes the right multiplication operator given by 
$$R_J(a)=aJ,\quad a\in \mathbb H.$$
\end{remark}

\begin{thm}
    Let $\Omega_1,\Omega_2$ be domains in $(\mathbb{H},\tau_s),$ and $f, g \in \mathcal{SR}(\Omega_1\times \Omega_2)$ with $f=g$ on  $\Omega_{1,\mathbb{R}}\times  \Omega_{2,\mathbb{R}}\neq \emptyset,$  where $\Omega_{\ell,\mathbb{R}}=\Omega_{\ell}\cap \mathbb{R}\ (\ell=1,2).$  Then $f=g$ on $\Omega_1\times \Omega_2.$  
\end{thm}

\begin{proof}
    Let $f,\ g:\Omega_{1}\times  \Omega_{2}\rightarrow \mathbb{H}$ be the two-sided slice regular functions. For any fixed $x\in \Omega_{1,\mathbb{R}},$ then the functions $f(x,q)$ and $g(x,q)$ are both right slice regular on $\Omega_2$ for all $q\in \Omega_2.$ Since $f(p,q)$ and $g(p,q)$ coincide on $\Omega_{1,\mathbb{R}}\times  \Omega_{2,\mathbb{R}},$  by the Identity principle of slice regular functions of one variable,
    \begin{equation}\label{2.2}
        f(x,q)=g(x,q)\quad \text{on}\ \Omega_{1,\mathbb R}\times \Omega_2.
    \end{equation}
For any fixed $q\in \Omega_{2},$ then the functions $f(p,q)$ and $g(p,q)$ are both left slice regular on $\Omega_1$ for all $p\in \Omega_1.$ According to formula (\ref{2.2}) and the Identity principle of slice regular functions of one variable,
    \begin{equation*}
        f(p,q)=g(p,q)\quad \text{for\ all }\ p\in \Omega_{1}.
    \end{equation*}
It follows that $f=g$ on $\Omega_1\times \Omega_2.$    
\end{proof}
\begin{defn}
   A set $U \subset \mathbb H$ is called {\bf axially symmetric} if $$U=\bigcup_{x+yI\in U}x+y\mathbb{S}.$$
\end{defn}
For convenience, denote
$$ \mathbb S^2_*:=\{(I_1,I_2)\in \mathbb S^2\ \ | \ I_1\neq I_2\}.$$
\begin{thm}\label{simple}$($Representation formula$)$
     Let $\Omega_1,\Omega_2$ be axially symmetric domains in $(\mathbb{H},\tau_s),$ and $f\in \mathcal{SR}(\Omega_{1}\times  \Omega_{2}).$ Let
$(p,q)=(x+Iy,u+Jv)\in \Omega_1\times \Omega_2.$ For all $(K_1,K_2),(L_1,L_2)\in \mathbb{S}^2_*,$ we have 
\begin{equation}\label{representation formula}
    f(p,q) =
  (1,I)\,A_K^{-1} M_fB_L^{-1}
  \binom{1}{J},
\end{equation}
where $$A_K:=\begin{pmatrix}
            1 & K_1\\ 1 & K_2
        \end{pmatrix}\quad \text{and}\quad B_L:=\begin{pmatrix}
            1 & 1\\ L_1& L_2
        \end{pmatrix},$$ and 
\[
  M_f:=
  \begin{pmatrix}
    f(x+yK_1,u+vL_1) & f(x+yK_1,u+vL_2)\\
    f(x+yK_2,u+vL_1) & f(x+yK_2,u+vL_2)
  \end{pmatrix}.
\]
\end{thm}

\begin{proof}
    Let
$(p,q)=(x+Iy,u+Jv)\in \Omega_1\times \Omega_2.$ For any fixed $p\in \Omega_1,\ f(p,q)$ is right slice regular with respect to $q$ on $\Omega_2$.
By the representation formula of one variable, for any $(L_1,L_2)\in \mathbb S_*^2,$ we have
\begin{equation}\label{r0}
    f(p,u+Jv) =
  \begin{pmatrix}
    f(p,u+L_1v) & f(p,u+L_2v)\\
  \end{pmatrix}\ 
  B_L^{-1}
  \binom{1}{J}.
\end{equation}
Since  $f(p,q)$ is left slice regular with respect to $p$ on $\Omega_1$, by the representation formula of one variable, for fixed $q\in \Omega_2$, we have
 $$f(x+Iy,q) =(1,I)A_K^{-1}
  \begin{pmatrix}
    f(x+K_1y,q) \\ f(x+K_2y,q)
  \end{pmatrix},\quad \forall (K_1,K_2)\in \mathbb S_*^2.$$
 It follows that 
 \begin{equation}\label{r1}
     f(x+Iy,u+L_1v)=(1,I)A_K^{-1}
  \begin{pmatrix}
    f(x+K_1y,u+L_1v) \\ f(x+K_2y,u+L_1v)
  \end{pmatrix}
 \end{equation} 
and
  \begin{equation}\label{r2}
     f(x+Iy,u+L_2v)=(1,I)A_K^{-1}
  \begin{pmatrix}
    f(x+K_1y,u+L_2v) \\ f(x+K_2y,u+L_2v)
  \end{pmatrix}.
 \end{equation}
Combining formulas (\ref{r0}),(\ref{r1}) and (\ref{r2}), the representation formula (\ref{representation formula}) is obtained.
\end{proof}

According to Theorem \ref{simple}, when $K_2=-K_1$ and $L_2=-L_1$, we can get following result.
\begin{corollary}\label{representation formula simple}
     Let $\Omega_1,\Omega_2$ be axially symmetric domains in $(\mathbb{H},\tau_s),$ and $f\in \mathcal{SR}(\Omega_{1}\times  \Omega_{2}).$ Let
$(p,q)=(x+Iy,u+Jv)\in \Omega_1\times \Omega_2.$ For all $K_1,L_1\in \mathbb{S},$  we have 
\begin{equation*}
    \begin{split}
        f(p,q)&=\frac{1}{4}\left[f(z^{K_1},w^{L_1})+f(z^{K_1},w^{-{L_1}})+f(z^{-{K_1}},w^{L_1})+f(z^{-{K_1}},w^{-{L_1}})\right]\\
        &-\frac{ I{K_1}}{4}\left[f(z^{K_1},w^{L_1})+f(z^{K_1},w^{-{L_1}})-f(z^{-{K_1}},w^{L_1})-f(z^{-{K_1}},w^{-{L_1}})\right]\\
        &-\frac{1}{4}\left[f(z^{K_1},w^{L_1})-f(z^{K_1},w^{-{L_1}})+f(z^{-{K_1}},w^{L_1})-f(z^{-{K_1}},w^{-{L_1}})\right]{L_1}J\\
        &+\frac{IK_1}{4}\left[f(z^{K_1},w^{L_1})-f(z^{K_1},w^{-{L_1}})-f(z^{-{K_1}},w^{L_1})+f(z^{-{K_1}},w^{-{L_1}})\right]{L_1}J,
    \end{split}
\end{equation*}
where $z^{\pm K_1}=x\pm K_1y$ and $w^{\pm L_1}=u\pm L_1v.$
\end{corollary}
This result shows that, for a two-sided slice regular function defined on an axially symmetric domain, the function values on the entire domain can be determined by its values on one full slices, since the four points $( z^{\pm K_1}, w^{\pm L_1}) $ are all on the slice $\mathbb C_{K_1}\times \mathbb C_{L_1}.$

 \begin{example}
 The right slice hyperholomorphic Cauchy kernel:
       $$S^{-1}_R(p,q):=-(q-\overline{p} )(q^2-2\operatorname{Re} (p)q+|p|^2)^{-1},\ q\notin [p],$$
       which is two-sided
       slice regular, where $[p]=\{\operatorname{Re}(p)+I |\operatorname{Im} (p)|:I\in \mathbb S\}$ denotes the 2-dimensional sphere. 
   \end{example}

\section{Two-sided slice functions}
Next, we introduce the definition of two-sided slice functions and some of their properties, and establish the equivalence between slice function with holomorphic stem functions and the slice regular functions on the axially symmetric domains discussed in Section 2. We also introduce the star-product of two slice regular functions and show that it preserves regularity.

Let $\Omega_1,\Omega_2\subset \mathbb H.$ Denote
     $$\Omega_1^i:=\{x+iy:x+Iy\in \Omega_1 \ \text{for some } I\in \mathbb{S}\}$$
     and
      $$\Omega_2^i:=\{u+iv:u+Jv\in \Omega_2 \ \text{for some } J\in \mathbb{S}\}.$$
Let $z=x+yi\in\mathbb{C}$ and $I\in\mathbb{S}$. Denote $$z^I:=x+yI.$$

\begin{defn}
Let $U_1,U_2\in \tau(\mathbb C)$ be symmetric with respect to the real axis, i.e.,
\[
\overline{U_\ell}=U_\ell,\qquad \ell=1,2.
\] A function 
 $$F= \begin{pmatrix}
    F_{11} & F_{12}\\
    F_{21} & F_{22}
  \end{pmatrix}:U_1\times U_2\rightarrow \mathbb{H}^{2\times 2}$$
is called intrinsic, if for each $\ell,t=1,2$ satisfies

\begin{equation}\label{Stem condition}
    \left\{
\begin{array}{ll}
F_{\ell,t}(\overline  z,w)= (-1)^{\ell-1}F_{\ell,t}(z,w)\\
F_{\ell,t}(z,\overline w)=(-1)^{t-1}F_{\ell,t}(z,w).
\end{array}\right.
\end{equation}
\end{defn}

\begin{remark}
   It is readily seen that \eqref{Stem condition} in the above definition is equivalent to
   \begin{equation}  \label{copynine}
   \left \{ 
    \begin{aligned}
    F(\overline z,w)= \begin{pmatrix}
    1 & 0\\
    0 & -1
  \end{pmatrix} F(z,w), \\
       F( z,\overline w)= F(z,w)\begin{pmatrix}
    1 & 0\\
    0 & -1
  \end{pmatrix}.
  \end{aligned}
   \right.
    \end{equation}
\end{remark}
\begin{prop}\label{nine-9}
Let $\Omega_1,\Omega_2\in\tau_s(\mathbb{H})$ be axially symmetric and a function $F(z,w):\Omega_{1}^i\times  \Omega_{2}^i\rightarrow \mathbb{H}^{2\times 2}$. Then the following statements are equivalent:
\begin{enumerate}[label=(\roman*)]
	\item $F$ is intrinsic.
        
    \item There is a function $f:\Omega_1\times\Omega_2\rightarrow\mathbb{H}$ such that
    \begin{equation}\label{induced}
        f(z^I,w^J):=
        (1,I)\,F(z,w)
        \binom{1}{J},\qquad\forall\ (z^I,w^J)\in\Omega_1\times \Omega_2.
    \end{equation}

    \item There is a unique function $f:\Omega_1\times\Omega_2\rightarrow\mathbb{H}$ such that \eqref{induced} holds.

    \item There is a function $f:\Omega_1\times\Omega_2\rightarrow\mathbb{H}$ such that
    \begin{equation}\label{eq-ind}
        F(z,w)=A_K^{-1}
        \begin{pmatrix}
            f(z^{K_1},w^{L_1}) & f(z^{K_1},w^{L_2})\\
            f(z^{K_2},w^{L_1}) & f(z^{K_2},w^{L_2})
        \end{pmatrix}
        B_L^{-1},
    \end{equation}
    for all $(z,w)\in\Omega_1^i\times \Omega_2^i$ and  $(K_1,K_2),(L_1,L_2)\in \mathbb{S}^2_*$.

    \item There is a unique function $f:\Omega_1\times\Omega_2\rightarrow\mathbb{H}$ such that \eqref{eq-ind} holds.
\end{enumerate}

We say that $f$ is the two-sided slice function of $F$ and $F$ is the stem function of $f$ if the above statements hold. And in this case, denote $$\mathcal{I}(F):=f.$$
\end{prop}

\begin{proof}
 $(i)\Longrightarrow (ii)$ Let $(z,w)=(x+iy,u+iv)\in \Omega_1^i\times \Omega_2^i,$ 
if there are $(z',w')\in \Omega_1^i\times \Omega_2^i,$
 and $I',J'\in \mathbb{S}$ such that $(z'^{I'},w'^{J'})=(z^I,w^J),$ then $x=x'$ and $u=u'.$ If $\{x,y,u,v\}$ and $\{x',y',u',v'\}$ are not exactly the same, we obtain 
\begin{equation*}
   \left\{\begin{array}{cc}
    y'=-y, I'=-I\\
     v'=v,J'=J.
    \end{array}\right. \quad \text{or} \quad
   \left\{\begin{array}{cc}
      y'=y, I'=I\\
     v'=-v,J'=-J.
    \end{array}\right.\quad \text{or}\quad
    \left\{\begin{array}{cc}
      y'=-y, I'=-I\\
     v'=-v,J'=-J.
    \end{array}\right.
\end{equation*}
In the first case,
\begin{equation*}
\begin{split}
    g(z',{I'},w',{J'})&:=
        (1,-I)\,F(\overline z,w) 
        \binom{1}{J}\\
        &=(1,I)\,\begin{pmatrix}
    1 & 0\\
    0 & -1
  \end{pmatrix}F(\overline z,w)
        \binom{1}{J}\\
        &=(1,I)\,F(z,w)
        \binom{1}{J}\\
        &=g(z,I,w,J).
\end{split}
\end{equation*}
The third equation above is because the function $F$ is intrinsic, i.e. $F$ satisfies formula (\ref{copynine}).
Similarly, in the other two cases, $$g(z,I,w,J)=g(z',{I'},w',{J'})$$
still holds true. Hence $f$ is well-defined on $\Omega_1\times\Omega_2,$ which is a function from $\Omega_1\times\Omega_2$ to $\mathbb{H}.$

$(ii)\Longrightarrow (iii)$ It is obvious.

$(iii)\Longrightarrow (iv)$ For all $(z,w)\in\Omega_1^i\times \Omega_2^i$ and $(K_1,K_2),(L_1,L_2)\in\mathbb{S}^2_*,$ we have
\begin{equation*}
        f(z^{K_{\ell}},w^{L_{t}})=
        (1,K_{\ell})\,F(z,w)
        \binom{1}{L_{t}}\quad \text{for}\ \ell,t=1,2.
    \end{equation*}
It follows that
\begin{equation*}
        \begin{pmatrix}
            f(z^{K_1},w^{L_1}) & f(z^{K_1},w^{L_2})\\
            f(z^{K_2},w^{L_1}) & f(z^{K_2},w^{L_2})
        \end{pmatrix}=A_K
        F(z,w)
       B_L.
    \end{equation*}
Thus
\begin{equation*}
        F(z,w)=A_K^{-1}
        \begin{pmatrix}
            f(z^{K_1},w^{L_1}) & f(z^{K_1},w^{L_2})\\
            f(z^{K_2},w^{L_1}) & f(z^{K_2},w^{L_2})
        \end{pmatrix}
       B_L^{-1}.
    \end{equation*}
    
$(iv)\Longrightarrow (v)$  If there is also a function  $g:\Omega_1\times\Omega_2\rightarrow\mathbb{H}$ such that
    \begin{equation*}
        F(z,w)=A_K^{-1}
        \begin{pmatrix}
            g(z^{K_1},w^{L_1}) & g(z^{K_1},w^{L_2})\\
            g(z^{K_2},w^{L_1}) & g(z^{K_2},w^{L_2})
        \end{pmatrix}
        B_L^{-1},
    \end{equation*}
    for all $(z,w)\in\Omega_1^i\times \Omega_2^i$ and  $(K_1,K_2),(L_1,L_2)\in \mathbb{S}^2_*$.
Then 
\begin{equation}\label{ninecopy}
       \begin{pmatrix}
            g(z^{K_1},w^{L_1}) & g(z^{K_1},w^{L_2})\\
            g(z^{K_2},w^{L_1}) & g(z^{K_2},w^{L_2})
        \end{pmatrix} =A_K F(z,w)
        B_L.
    \end{equation} 
Substitute \eqref{eq-ind}
into \eqref{ninecopy} we have
\begin{equation}\label{equal}
       \begin{pmatrix}
            g(z^{K_1},w^{L_1}) & g(z^{K_1},w^{L_2})\\
            g(z^{K_2},w^{L_1}) & g(z^{K_2},w^{L_2})
        \end{pmatrix} =\begin{pmatrix}
            f(z^{K_1},w^{L_1}) & f(z^{K_1},w^{L_2})\\
            f(z^{K_2},w^{L_1}) & f(z^{K_2},w^{L_2})
        \end{pmatrix}.
    \end{equation} 
 Thus
\begin{equation*}
        g(z^I,w^J)=
        f(z^I,w^J),\quad  \ \forall\ (z^I,w^J) \in \Omega_1\times\Omega_2.
    \end{equation*}

$(v)\Longrightarrow (i)$ If there is a unique function $f:\Omega_1\times\Omega_2\rightarrow\mathbb{H}$ such that \eqref{eq-ind} holds
    for all $(z,w)\in\Omega_1^i\times \Omega_2^i$ and  $(K_1,K_2),(L_1,L_2)\in \mathbb{S}^2_*$. Denote $A_{-K}:=\begin{pmatrix}
            1 & -K_1\\ 1 & -K_2
        \end{pmatrix},$ and $B_{-L}:= \begin{pmatrix}
            1 & 1\\- L_1& -L_2
        \end{pmatrix},$
then
\begin{equation*}
    \begin{split}
        F(\overline z,w)=&A_K^{-1}
        \begin{pmatrix}
            f(\overline z^{K_1},w^{L_1}) & f(\overline z^{K_1},w^{L_2})\\
            f(\overline z^{K_2},w^{L_1}) & f(\overline z^{K_2},w^{L_2})
        \end{pmatrix}
       B_L^{-1}\\
        =&A_K^{-1}
        \begin{pmatrix}
            f(z^{-K_1},w^{L_1}) & f( z^{-K_1},w^{L_2})\\
            f( z^{-K_2},w^{L_1}) & f( z^{-K_2},w^{L_2})
        \end{pmatrix}
        B_L^{-1}\\
        =&\begin{pmatrix}
            1 & 0\\ 0 & -1
        \end{pmatrix}{A_{-K}}^{-1}
        \begin{pmatrix}
            f(z^{-K_1},w^{L_1}) & f( z^{-K_1},w^{L_2})\\
            f( z^{-K_2},w^{L_1}) & f( z^{-K_2},w^{L_2})
        \end{pmatrix}
        B_L^{-1}\\
        =&\begin{pmatrix}
            1 & 0\\ 0 & -1
        \end{pmatrix}F(z,w).
    \end{split}
\end{equation*}
Similarly, we obtain
$$F(z,\overline w)=F(z,w)\begin{pmatrix}
            1 & 0\\ 0 & -1
        \end{pmatrix}. $$
Thus $F$ satisfies \eqref{copynine}, which is intrinsic.
\end{proof}
Since quaternionic multiplication is non-commutative, the complex structures associated with the two variables need to be distinguished. To this end, we introduce two independent complex tensor factors and set
$$
\mathcal{A}:=\mathbb{H}^{2\times 2}\otimes\mathbb{C}\otimes\mathbb{C}.
$$
Consider the linear isomorphism
$$
\psi:\mathbb{H}^{2\times2}\longrightarrow\mathcal{A}
$$
defined by
$$
\psi\left(
\begin{array}{ll}
\mu_{11} & \mu_{12}\\
\mu_{21} & \mu_{22}
\end{array}
\right)
=
\mu_{11}\otimes1\otimes1
+\mu_{12}\otimes1\otimes i
+\mu_{21}\otimes i\otimes1
+\mu_{22}\otimes i\otimes i.
$$

The following lemma is immediate.
\begin{lem}\label{tensor product}
    Let $\mu=\left(\begin{array}{ll}\mu_{11} & \mu_{12} \\ \mu_{21} & \mu_{22}\end{array}\right)\in \mathbb{H}^{2\times 2},$ then
$$(1\otimes i\otimes 1)\psi (\mu)=\psi \left(\begin{pmatrix}
    0& -1\\
    1 & 0
  \end{pmatrix}\mu\right)$$
 and
 $$(1\otimes 1\otimes i)\psi (\mu)=\psi \left(\mu \begin{pmatrix}
    0& 1\\
    -1 & 0
  \end{pmatrix}\right).$$
\end{lem}

This allows us to formulate the Cauchy--Riemann equations for quaternion-valued functions of two complex variables as follows.
\begin{defn}\label{Fholo}
Let \(U_1,U_2\in\tau(\mathbb{C})\) be axially symmetric. A real differential function
$$
F=
\begin{pmatrix}
F_{11}&F_{12}\\
F_{21}&F_{22}
\end{pmatrix}
:U_1\times U_2\longrightarrow\mathbb{H}^{2\times2}
$$
is called holomorphic if, for every
$
(z,w)=(x+iy,u+iv)\in U_1\times U_2,
$
it satisfies
\begin{equation}\label{nine-4}
\left\{
\begin{aligned}
\frac12\left(
\frac{\partial}{\partial x}
+(1\otimes i\otimes1)\frac{\partial}{\partial y}
\right)\psi(F(z,w))&=0,\\
\frac12\left(
\frac{\partial}{\partial u}
+(1\otimes1\otimes i)\frac{\partial}{\partial v}
\right)\psi(F(z,w))&=0.
\end{aligned}
\right.
\end{equation}
\end{defn}

\begin{remark}
   It is easily verified using Lemma \ref{tensor product} that \eqref{nine-4} is equivalent to the components of $F$ satisfying the following Cauchy-Riemann equations:
\begin{equation}
     \left\{\begin{array}{cc}
      \frac{\partial}{\partial x}F_{1\ell }(z,w)=\frac{\partial}{\partial y}F_{2\ell }(z,w)\\ \\
       \frac{\partial}{\partial y}F_{1\ell }(z,w)=-\frac{\partial}{\partial x}F_{2\ell }(z,w)
    \end{array}\right.\quad \text{and}\quad  \left\{\begin{array}{cc}
      \frac{\partial}{\partial u}F_{\ell 1}(z,w)=\frac{\partial}{\partial v}F_{\ell 2}(z,w)\\ \\
       \frac{\partial}{\partial v}F_{\ell 1}(z,w)=-\frac{\partial}{\partial u}F_{\ell 2}(z,w)
    \end{array} \right.
\end{equation}
    for $\ell=1,2.$
\end{remark}

\begin{prop}\label{F-psi}
    Let $U_1,U_2\in\tau(\mathbb{C})$ be axially symmetric, let $F,\ G:U_1\times U_2\rightarrow \mathbb{H}^{2\times 2}$ are holomorphic functions. Then the function
   $\psi^{-1}(\psi(F)\psi(G))$ is also holomorphic on $U_1\times U_2.$
\end{prop}
\begin{proof}
    Let $(z,w)=(x+iy,u+i v )\in U_1\times U_2$, since $F,   G$ are holomorphic functions, we can readily obtain
    \begin{equation*}
        \frac{1}{2}\left(\frac{\partial}{\partial x}+(1\otimes i\otimes1)\frac{\partial}{\partial y}\right)(\psi(F(z,w))\psi(G(z,w)))\\
        =0
    \end{equation*}
    and
    $$\frac{1}{2}\left(\frac{\partial}{\partial u}+(1\otimes 1\otimes i)\frac{\partial}{\partial v}\right)(\psi(F(z,w))\psi(G(z,w)))=0. $$
    Thus $\psi^{-1}(\psi(F)\psi(G))$ is holomorphic on $U_1\times U_2.$
\end{proof}

\begin{thm}\label{f-F}
    Let $\Omega_1,\Omega_2\in\tau_s(\mathbb{H})$ be axially symmetric, and $f:\Omega_{1}\times  \Omega_{2}\rightarrow \mathbb{H}$  be a two-sided slice function induced by a stem function $F:\Omega_1^i\times\Omega_2^i\rightarrow\mathbb{H}^{2\times 2}.$ Then $f$ is slice regular if and only if $F$ is holomorphic.
\end{thm}

\begin{proof}
   ``$\Rightarrow$'' Let $(z,w)=(x+iy,u+iv)\in \Omega_1^i\times \Omega_2^i$ and $(K_1,K_2),(L_1,L_2)\in \mathbb{S}^2_*.$ By Proposition \ref{nine-9},
    \begin{equation}\label{nine_1}
         F(z,w)=A_K^{-1}
        \begin{pmatrix}
            f(z^{K_1},w^{L_1}) & f(z^{K_1},w^{L_2})\\ 
            f(z^{K_2},w^{L_1}) & f(z^{K_2},w^{L_2})
        \end{pmatrix}
        B_L^{-1}
    \end{equation}
    is real differential, since $f$ is.
         Let $\sigma$ be the matrix  $\begin{pmatrix}
  0  & -1\\
    1 & 0
  \end{pmatrix} ,$
according to Lemma \ref{tensor product},
\begin{equation}\label{nine_2}
\begin{split}
    &\left(\frac{\partial}{\partial x}+(1\otimes i\otimes1)\frac{\partial}{\partial y}\right)\psi(F(z,w))=
   \psi\left(F_x(z,w) + \sigma F_y(z,w)\right),
\end{split}
\end{equation}
where $F_x(z,w):=\frac{\partial}{\partial x}F(z,w),$ and $ F_y(z,w):=\frac{\partial}{\partial y}F(z,w).$ 
It follows from \eqref{nine_1} that
\begin{equation}\label{nine_33}
   F_x(z,w)= A_K^{-1}
        \begin{pmatrix}
            f_x(z^{K_1},w^{L_1}) & f_x(z^{K_1},w^{L_2})\\ 
            f_x(z^{K_2},w^{L_1}) & f_x(z^{K_2},w^{L_2})
        \end{pmatrix}
        B_L^{-1},
\end{equation}
and
\begin{equation}\label{nine_3}
    \begin{split}
       \sigma F_y(z,w)
       =& \sigma A_K^{-1}
        \begin{pmatrix}
            f_y(z^{K_1},w^{L_1}) & f_y(z^{K_1},w^{L_2})\\ 
            f_y(z^{K_2},w^{L_1}) & f_y(z^{K_2},w^{L_2})
        \end{pmatrix}
    B_L^{-1}\\=&A_K^{-1}\begin{pmatrix}
    K_1&0\\0&K_2
\end{pmatrix} \begin{pmatrix}
            f_y(z^{K_1},w^{L_1}) & f_y(z^{K_1},w^{L_2})\\ 
            f_y(z^{K_2},w^{L_1}) & f_y(z^{K_2},w^{L_2})
        \end{pmatrix}
    B_L^{-1}\\
       =&A_K^{-1}
        \begin{pmatrix}
           K_1 f_y(z^{K_1},w^{L_1}) & K_1f_y(z^{K_1},w^{L_2})\\ 
            K_2f_y(z^{K_2},w^{L_1}) & K_2f_y(z^{K_2},w^{L_2})
        \end{pmatrix}
       B_L^{-1},
    \end{split}
\end{equation}
where $f_x$ and $f_y$ are the partial derivatives of $f$ w.r.t $x$ and $y$, respectively.  Denote 
$$\overline \partial_{K_{\ell}}^L:=\frac{1}{2}\left(\frac{\partial}{\partial x}
    +K_{\ell}\frac{\partial}{\partial y}\right)\quad  \text{and}\quad \overline \partial_{L_{\ell}}^R:=\frac{1}{2}\left(\frac{\partial}{\partial u}
    +\frac{\partial}{\partial v}R_{L_{\ell}}\right)$$ for $ \ell=1,2.$
Substitute \eqref{nine_33}and \eqref{nine_3} into \eqref{nine_2}, we obtain
\begin{equation}\label{nine_4}
    \begin{split}
 &\frac{1}{2}\left(\frac{\partial}{\partial x}+(1\otimes i\otimes1)\frac{\partial}{\partial y}\right)\psi(F(z,w))\\
    =&\psi \left [A_K^{-1}\begin{pmatrix}
            \overline \partial_{K_1}^Lf(z^{K_1},w^{L_1}) & \overline \partial_{K_1}^Lf(z^{K_1},w^{L_2})\\ \\
            \overline \partial_{K_2}^Lf(z^{K_2},w^{L_1}) & \overline \partial_{K_2}^Lf(z^{K_2},w^{L_2})
        \end{pmatrix}B_L^{-1} \right ].
\end{split}
\end{equation}
Since $f\in \mathcal{SR}(\Omega_1\times \Omega_2)$, it follows that
\begin{equation}\label{4plane}
    \overline \partial_{K_{\ell}}^Lf(z^{K_{\ell}},w^{L_{\ell}})=0\quad \text{and} \quad \overline \partial_{L_{\ell}}^Rf(z^{K_{\ell}},w^{L_{\ell}})=0,\ \ \ell=1,2.
\end{equation}
Combining \eqref{nine_4} and \eqref{4plane}, we have 
$$\frac{1}{2}\left(\frac{\partial}{\partial x}+(1\otimes i\otimes1)\frac{\partial}{\partial y}\right)\psi(F(z,w))=0.$$
Similarly, we obtain
\begin{equation}\label{F-R}
\begin{split}
    &\frac{1}{2}\left(\frac{\partial}{\partial u}+(1\otimes 1\otimes i)\frac{\partial}{\partial v}\right)\psi(F(z,w))\\
    &=\psi \left [A_K^{-1}\begin{pmatrix}
            \overline \partial_{L_1}^Rf(z^{K_1},w^{L_1}) & \overline \partial_{L_1}^Rf(z^{K_1},w^{L_2})\\\\
           \overline \partial_{L_2}^Rf(z^{K_2},w^{L_1}) & \overline \partial_{L_2}^Rf(z^{K_2},w^{L_2})
        \end{pmatrix}B_L^{-1} \right ]\\
        &=0.
\end{split}
\end{equation}
Therefore, \(F\) satisfies \eqref{nine-4}, which implies that \(F\) is holomorphic.

 ``$\Leftarrow$ '' let $f=\mathcal{I}(F)$ and $F$ is holomorphic, then \eqref{nine-4} follows. From Proposition \ref{nine-9}, we obtain
$$f(z^I,w^J)=(1,I)F(z,w)\binom{1}{J}$$
for all $(z,w)=(x+iy,u+iv)\in \Omega_1^i\times \Omega_2^i.$ Note that
$I(1,I)=(1,I)\sigma,$ hence
\begin{equation*}
    \begin{split}
         &\frac{1}{2}\left(\frac{\partial}{\partial x}
    +I\frac{\partial}{\partial y}\right)f(z^I,w^J)\\
    =&\frac{1}{2}\left[(1,I)F_x(z,w)\binom{1}{J}+I(1,I)F_y(z,w)\binom{1}{J}\right]\\
    =&\frac{1}{2}(1,I)\left (F_x(z,w)+\sigma F_y(z,w)\right)\binom{1}{J}\\
    =&\frac{1}{2}(1,I)\left (F_x(z,w)+\psi^{-1}((1\otimes i\otimes 1)F_y(z,w))\right)\binom{1}{J}
    \\
    =&(1,I)\psi ^{-1} \left[\frac{1}{2}\left(\frac{\partial}{\partial x}
    +(1\otimes i\otimes1))\frac{\partial}{\partial y}\right)\psi( F(z,w)) \right]\binom{1}{J}\\
    =& 0,
    \end{split}
\end{equation*}
where the third and last equalities follow from Lemma \ref{tensor product} and \eqref{nine-4}, respectively.
Similarly, we can prove
\begin{equation*}
  \frac{1}{2}\left(\frac{\partial}{\partial u}
    +\frac{\partial}{\partial v}R_J\right)f(z^I,w^J)= 0.
\end{equation*}
Hence $f$ satisfies \eqref{tsr}, implying that $f$ is a two-sided slice regular functions.
\end{proof}

\begin{corollary}\label{nine-slice}

     Let $\Omega_1,\Omega_2$ be axially symmetric domains in $\tau_s(\mathbb{H})$, and $f: \Omega_1 \times \Omega_2 \rightarrow \mathbb{H}$  be a two-sided slice regular function. Let $(z,w)\in \Omega_1^i\times \Omega_2^i$, and $(K_1, K_2)$, $(L_1, L_2 )\in \mathbb{S}^2_*$. Let
     \begin{equation}\label{C_fdef}
C_f(z, w)=A_K^{-1}\left(\begin{array}{ll}
f\left(z^{K_1}, w^{L_1}\right) & \left.f(z^{K_1}, w^{L_2}\right) \\
f\left(z^{K_2}, w^{L_1}\right) & f\left(z^{K_2}, w^{L_2}\right)
\end{array}\right) B_L^{-1} .
\end{equation}
Then $C_f(z, w)$ does not depend on the choice of $K_1, K_2, L_1, L_2$. Moreover, $C_f(z, w)$ is the two-sided holomorphic stem function of $f.$
\end{corollary}

\begin{proof}

Let  $(z,w) =( x+iy,u+iv)\in \Omega_1^i\times  \Omega_2^i,$ and 
$$ \lambda :=A_K^{-1}\left(\begin{array}{ll}
f\left(z^{K_1}, w^{L_1}\right) & \left.f(z^{K_1}, w^{L_2}\right) \\
f\left(z^{K_2}, w^{L_1}\right) & f\left(z^{K_2}, w^{L_2}\right)
\end{array}\right) B_L^{-1}.$$
 Since $f$ is two-sided slice regular, by formula (\ref{representation formula}), for any $I,J\in \mathbb S,$ and $(K_1,K_2),(L_1,L_2)\in \mathbb{S}_*^2$, we have
\begin{equation}\label{C_f}
        f(z^I, w^J)=\left(\begin{array}{ll}1 & I\end{array}\right) \lambda \binom{1}{J}.
\end{equation}
Let $ (K_1',K_2'),(L_1',L_2')\in \mathbb{S}^2_*,$ from above formula, it follows that
$$\left(\begin{array}{ll}
f\left(z^{K_1'}, w^{L_1'}\right) & \left.f(z^{K_1'}, w^{L_2'} \right) \\
f\left(z^{K_2'}, w^{L_1'} \right) & f\left(z^{K_2'}, w^{L_2'}\right)
\end{array}\right) =A_{K'}\lambda
 B_{L'}. $$
Then
$$
{A_{K'}}^{-1}\left(\begin{array}{ll}
f\left(z^{K_1'}, w^{L_1'}\right) & \left.f(z^{K_1'}, w^{L_2'} \right) \\
f\left(z^{K_2'}, w^{L_1'} \right) & f\left(z^{K_2'}, w^{L_2'}\right)
\end{array}\right) {B_{L'}}^{-1}=\lambda,
$$
which implies that $C_f(z, w)$ does not depend on the choice of $K_1, K_2, L_1, L_2$.
Moreover, based on 
$$A_K^{-1} =\begin{pmatrix}
    1 & 0\\
    0& -1
  \end{pmatrix}{A_{-K}}^{-1}$$
  and
$$B_L^{-1} ={B_{-L}}^{-1}\begin{pmatrix}
    1 & 0\\
    0& -1
  \end{pmatrix},$$
 it is easy to verify that
$$C_f(\overline z, w)=\begin{pmatrix}
    1 & 0\\
    0& -1
  \end{pmatrix}C_f(z, w)\quad \text{and}\quad C_f( z,\overline w)=C_f(z, w)\begin{pmatrix}
    1 & 0\\
    0& -1
  \end{pmatrix}. $$
 Thus $C_f(z, w)$ satisfies (\ref{Stem condition}), which is a stem function and 
 $f$ is induced by $C_f(z, w)$ due to \eqref{C_f}. Since $f$ is two-sided slice regular, by Theorem \ref{f-F}, we have that $C_f(z, w)$ is also holomorphic.
\end{proof}
\begin{remark}
    From Corollary \ref{nine-slice}, we know that a slice regular function on an axially symmetric domain is necessarily a two-sided slice function.
\end{remark}
To study the star-product of stem functions and its compatibility with the stem condition, we introduce the following two real-linear operators.
\begin{defn}
    Define two real linear operators on algebra $\mathcal{A}=\mathbb{H}\otimes_{\mathbb{R}}\mathbb{C}\otimes_{\mathbb{R}}\mathbb{C}$ as follows:
    $$T_1(a\otimes \xi \otimes \eta)=a\otimes \overline \xi \otimes \eta,\quad T_2(a\otimes \xi \otimes \eta)=a\otimes  \xi \otimes \overline\eta.$$
\end{defn}
It is obvious that $T_1$ and $T_2$ are both  algebraic automorphisms on $\mathcal{A}$, and $T_1T_2=T_2T_1.$
\begin{lem}\label{stem-operator}
    The function $F:\Omega_1^i\times \Omega_2^i\rightarrow \mathbb{H}^{2\times 2}$ is a stem function if and only if
    $$\psi(F(\overline z,w))=T_1(\psi(F(z,w))),\quad  \psi(F( z,\overline w))=T_2(\psi(F(z,w))).$$
\end{lem}
\begin{proof}
    Let $F=\left(\begin{array}{ll}F_{11} & F_{12} \\ F_{21} & F_{22}\end{array}\right),$ then
    $$\psi(F)=F_{11}\otimes 1\otimes 1+F_{12} \otimes 1\otimes i+F_{21}\otimes i\otimes 1+ F_{22}\otimes i\otimes i.$$
Hence
\begin{equation*}
        T_1(\psi (F)) 
        =F_{11} \otimes 1\otimes 1+F_{12}  \otimes 1\otimes i-F_{21} \otimes i\otimes 1- F_{22} \otimes i\otimes i
\end{equation*}
and
\begin{equation*}
  T_2(\psi (F)) 
        =F_{11} \otimes 1\otimes 1-F_{12}  \otimes 1\otimes i+F_{21} \otimes i\otimes 1- F_{22} \otimes i\otimes i.
\end{equation*}
It follows that 
$$ T_1(\psi (F)(z,w))=\psi(F(\overline z,w))\quad \text{and}\quad T_2(\psi (F)(z,w))=\psi(F( z,\overline w))$$
if and only if 
$F$ satisfies formula \eqref{Stem condition}, which is a two-sided stem function.
\end{proof}

 \begin{lem}\label{FG}
     Let $F,G:\Omega_1^i\times \Omega_2^i\rightarrow \mathbb{H}^{2\times 2}$ be two-sided stem  functions, then 
     $$F*G=\psi^{-1}(\psi (F) \psi (G)) $$
     is also a two-sided stem  function on $\Omega_1^i\times \Omega_2^i.$
 \end{lem}

\begin{proof}
   By lemma \ref{stem-operator}, we have
\begin{equation*}
   \left\{\begin{array}{cc}
     \psi(F(\overline z,w))=T_1(\psi(F(z,w)))\\ \\
       \psi(F( z,\overline w))=T_2(\psi(F(z,w)))
    \end{array}\right.
and
\left\{\begin{array}{cc}
     \psi(G(\overline z,w))=T_1(\psi(G(z,w)))\\ \\
       \psi(G( z,\overline w))=T_2(\psi(G(z,w))).
    \end{array}\right.
\end{equation*}
Then for any $(z,w)\in \Omega_1^i\times  \Omega_2^i,$
\begin{equation*}
    \begin{split}
     \psi(F*G)(\overline z,w)&= (\psi( F)\psi (G)) (\overline z,w)\\
     &=(T_1\circ \psi(F))(T_1\circ \psi(G))(z,w)\\
     &=T_1\circ (\psi(F)\psi(G))(z,w)\\
     &=T_1\circ (\psi(F*G))(z,w)
    \end{split}
\end{equation*}
and
\begin{equation*}
    \begin{split}
     \psi(F*G)( z,\overline w)&= (\psi( F)\psi (G)) (z,\overline w)\\
     &=(T_2\circ \psi(F))(T_2\circ \psi(G))(z,w)\\
     &=T_2\circ (\psi(F)\psi(G))(z,w)\\
     &=T_2\circ (\psi(F*G))(z,w).
    \end{split}
\end{equation*}
Therefore, $F*G$ is a two-sided stem  function on $\Omega_1^i\times \Omega_2^i$ by Lemma \ref{stem-operator}. 
\end{proof}
\begin{defn}\label{nine12}
     Let $F,\ G:\Omega_1^i\times \Omega_2^i\rightarrow \mathbb{H}^{2\times 2}$ be two-sided stem functions and $f=\mathcal{I}(F),\ g=\mathcal{I}(G).$ The star-product of $f$ and $g$ is given by:
     $$f*g:=\mathcal{I}(F*G).$$
\end{defn}

\begin{thm}
     Let $\Omega_1,\Omega_2\in\tau_s(\mathbb{H})$ be axially symmetric, and $f=\mathcal{I}(F),\ g=\mathcal{I}(G):\Omega_{1}\times  \Omega_{2}\rightarrow \mathbb{H}$  be  two-sided slice regular functions, then $f*g$ is also a  two-sided slice regular functions on $\Omega_{1}\times  \Omega_{2}.$
\end{thm}
\begin{proof}
Since $f$ and $g$ 
    are two-sided slice regular, it follows from Theorem \ref{f-F} that $F,\ G$ are holomorphic. Due to Proposition \ref{F-psi}, $F*G$ is a holomorphic function. Therefore, $f*g$ is two-sided slice regular 
     by Lemma \ref{f-F}.
\end{proof}
Let \(B(0,r)\) denote the ball in $\mathbb H$ centered at \(0\) with radius \(r\). We have the following result.
\begin{thm}
 Let
$$
f(p,q)=\sum_{\alpha,\beta\in\mathbb N}p^\alpha a_{\alpha\beta}q^\beta,
\qquad
g(p,q)=\sum_{\alpha,\beta\in\mathbb N}p^\alpha b_{\alpha\beta}q^\beta
$$
be two convergent power series in $B(0,r_1)\times B(0,r_2)$. Then their star-product is the power series
$$
(f*g)(p,q)
=
\sum_{\alpha,\beta\in\mathbb N}
p^\alpha
\left(
\sum_{\substack{m+\ell=\alpha\\ n+t=\beta}}
a_{mn}b_{\ell t}
\right)
q^\beta
$$
in $B(0,r_1)\times B(0,r_2)$.
\end{thm}

\begin{proof}
Let $(z,w)=(x+iy,u+iv)$ be such that $(p,q)=(z^I,w^J)\in \mathbb B(0,r_1)\times \mathbb B(0,r_2)$ for some $I,J\in \mathbb S.$ Let $f=\mathcal{I}(F)$ and $ g=\mathcal{I}(G).$ Then 
$$\psi(F)(z,w)=\sum_{\alpha , \beta \in N} z_1^\alpha a_{\alpha \beta} w_1^\beta\quad \text{and}\quad  \psi(G)(z,w)=\sum_{\alpha , \beta \in N} z_1^\alpha b_{\alpha \beta} w_1^\beta,$$
where $(z_1,w_1)=\left (x+(1\otimes i\otimes 1)y, u+(1\otimes 1\otimes i)v\right ).$
Then 
$$\psi(F)\psi(G)=\sum_{\alpha, \beta \in \mathbb{N}} z_1^\alpha\left(\sum_{\substack{m+\ell=\alpha \\
n+t=\beta}} a_{m n} b_{\ell t}\right) w_1^\beta.$$
From Definition \ref{nine12} and Lemma \ref{FG}, it follows that
$$f * g(p, q)=\mathcal{I}\left (\psi^{-1}(\psi(F)\psi(G))\right )=\sum_{\alpha, \beta \in \mathbb{N}} p^\alpha\left(\sum_{\substack{m+\ell=\alpha \\
n+t=\beta}} a_{m n} b_{\ell t}\right) q^\beta.$$
\end{proof}

\section{Cauchy integral formula}
For slice regular functions of one quaternionic variable, the Cauchy integral formula determines a function on a symmetric slice domain from its values on a single slice. In this section, we establish the corresponding formula for two-sided slice regular functions.
\begin{thm}
  Let $f\in \mathcal{SR}(\Omega_1\times \Omega_2)$ and  $I,J\in \mathbb{S}$. Let $U_1$ and $U_2$ be bounded axially symmetric open subsets of $\Omega_{1,I}$ and $\Omega_{2,J}$ respectively, satisfying
   $$\overline{U}_1\in \Omega_{1,I}\quad \text{and}\quad  \overline{U}_2\in \Omega_{2,J}.$$
   If $\partial U_{1}$ and $\partial U_{2}$  are two finite union of disjoint rectifiable Jordan curves, respectively, then for any $(p,q)=(z^{K_1},w^{L_1})$ with $ K_1,L_1\in \mathbb S,$ and $(z,w)\in U_1^i\times U_2^i,$ 
  $$f(p,q)=\frac{1}{(2\pi)^2}\int_{\partial U_{1}\times { \partial{U_{2}}}}(s-p)^{-*}_lds_I f(s,t)dt_J(t-q)^{-*}_r,$$
  where 
$$(s-p)_{ l}^{-*}=\frac{1}{2} \left (\frac{1}{s-z^I}+\frac{1}{s-z^{-I}}\right ) +\frac{K_1I }{2}\left (\frac{1}{s-z^{-I}}-\frac{1}{s-z^I}\right ),$$
and
$$(t-q)^{-*}_{r}=\frac{1}{2} \left (\frac{1}{t-w^J}+\frac{1}{t-w^{-J}}\right ) +\left (\frac{1}{t-w^{-J}}-\frac{1}{t-w^J}\right )\frac{JL_1}{2}.$$
  
\end{thm}

\begin{proof} 
  Let $(p,q)=(z^{K_1},w^{L_1})$ with $ K_1,L_1\in \mathbb S,$ and $(z,w)\in U_1^i\times U_2^i.$ Since $f\in \mathcal{SR}(\Omega_1\times \Omega_2)$,  by Corollary \ref{representation formula simple}, we have
\begin{equation}\label{nine-cake}
    \begin{split}
        f(p,q)&=\frac{1}{4}\left[f(z^I,w^J)+f(z^I,w^{-J})+f(z^{-I},w^J)+f(z^{-I},w^{-J})\right]\\
        &-\frac{K_1 I}{4}\left[f(z^I,w^J)+f(z^I,w^{-J})-f(z^{-I},w^J)-f(z^{-I},w^{-J})\right]\\
        &-\frac{1}{4}\left[f(z^I,w^J)-f(z^I,w^{-J})+f(z^{-I},w^J)-f(z^{-I},w^{-J})\right]JL_1\\
        &+\frac{K_1I}{4}\left[f(z^I,w^J)-f(z^I,w^{-J})-f(z^{-I},w^J)+f(z^{-I},w^{-J})\right]JL_1.
    \end{split}
\end{equation}
Denote $(p_1,q_1):=(z^I,w^J),$ for each fixed $q_1,$
$f(p_1,q_1)$ is left slice regular with respect to $p_1$.
By the Cauchy integral formula in slice analysis,
\begin{equation}\label{Cauchy1}
      f(p_1,q_1)=\frac{1}{2 \pi} \int_{\partial U_1} \frac{1}{s-p_1} d s_{I} f(s, q_1),
\end{equation}
where $ds_I=-Ids.$
Notice that the function $f(s, q_1)$ is right slice regular with respect to $q_1$,  we have
\begin{equation}\label{right slice}
f(s, q_1)=\frac{1}{2 \pi} \int_{\partial U_2} f(s, t) d t_{J} \frac{1}{t-q_1},    
\end{equation}
where $dt_J=-Jdt.$
Substituting formula \eqref{right slice} into formula \eqref{Cauchy1}, we obtain
$$f(p_1,q_1)=\frac{1}{(2 \pi)^2} \int_{\partial U_1\times \partial U_2} \frac{1}{s-p_1} d s_I f(s, t) d t_J \frac{1}{t-q_1}.$$
That is,
$$f(z^I,w^J)=\frac{1}{(2 \pi)^2} \int_{\partial U_1\times \partial U_2} \frac{1}{s-z^I} d s_I f(s, t) d t_J \frac{1}{t-w^J}.$$
More generally,
$$f(z^{\pm I},w^{\pm J})=\frac{1}{(2 \pi)^2} \int_{\partial U_1\times \partial U_2} \frac{1}{s-z^{\pm I}} d s_I f(s, t) d t_J \frac{1}{t-w^{\pm J}}.$$
From \eqref{nine-cake}, it follows that
\begin{equation*}
f(p,q)=\frac{1}{(2 \pi)^2} \int_{\partial U_1\times \partial U_2} (s-p)_{ l}^{-*} ds_If(s, t) dt_J (t-q)^{-*}_{ r}.
\end{equation*}
\end{proof}

\begin{remark}
    A natural continuation of this work is the study of two-sided slice regular functions on non-axially symmetric domains. The corresponding extension results, representation formulas, and Cauchy-type integral formulas are currently under investigation.
\end{remark}

\bibliographystyle{plainnat}
\bibliography{mybibfile}

@article{GentiliStruppa2006,
  author  = {Graziano Gentili and Daniele C. Struppa},
  title   = {A new approach to {C}ullen-regular functions of a quaternionic variable},
  journal = {C. R. Math.},
  volume  = {342},
  number  = {10},
  pages   = {741--744},
  year    = {2006},
  doi     = {10.1016/j.crma.2006.03.015}
}

@article{gentili2007new4,
  title     = {A New Theory of Regular Functions of a Quaternionic Variable},
  author    = {Gentili, Gianluca and Struppa, Daniel C.},
  journal   = {Adv. Math.},
  volume    = {216},
  number    = {1},
  pages     = {279--301},
  year      = {2007},
  publisher = {Elsevier}
}

@article{ghiloni2011slice,
  author  = {Ghiloni, Riccardo and Perotti, Alessandro},
  title   = {Slice regular functions on real alternative algebras},
  journal = {Adv. Math.},
  volume  = {226},
  number  = {2},
  pages   = {1662--1691},
  year    = {2011}
}

@book{colombo2018spectral,
  title={Spectral theory on the {S}-spectrum for quaternionic operators},
  author={Colombo, Fabrizio and Gantner, Jonathan and Kimsey, David P and others},
  year={2018},
  publisher={Springer}
}

@article{colombo2007functional,
  author    = {Colombo, Fabrizio and Gentili, Graziano and Sabadini, Irene and Struppa, Daniele C.},
  title     = {A Functional Calculus in a Non Commutative Setting},
  journal   = {Electron. Res. Announc. Math. Sci.},
  volume    = {14},
  pages     = {60--68},
  year      = {2007},
  publisher = {American Institute of Mathematical Sciences}
}

@article{Colombo2009SliceMonogenic,
  author  = {Colombo, Fabrizio and Sabadini, Irene and Struppa, Daniele C.},
  title   = {Slice Monogenic Functions},
  journal = {Isr. J. Math.},
  volume  = {171},
  pages   = {385--403},
  year    = {2009},
  doi     = {10.1007/s11856-009-0055-4}
}

@article{gentili2010regular5,
  title={Regular functions on the space of {C}ayley numbers},
  author={Gentili, Gianluca and Struppa, Daniel C.},
  journal={Rocky Mountain J. Math.},
  pages={225--241},
  year={2010},
  publisher={JSTOR}
}

@article{ColomboSabadini2009,
  author  = {Colombo, Fabrizio and Sabadini, Irene},
  title   = {On Some Properties of the Quaternionic Functional Calculus},
  journal = {Journal of Geometric Analysis},
  volume  = {19},
  number  = {3},
  pages   = {601--627},
  year    = {2009},
  doi     = {10.1007/s12220-009-9075-x}
}

@article{ColomboSabadini2010lr,
  author  = {Colombo, Fabrizio and Sabadini, Irene},
  title   = {On the formulations of the quaternionic functional calculus},
  journal = {Journal of Geometry and Physics},
  volume  = {60},
  number  = {10},
  pages   = {1490--1508},
  year    = {2010},
  doi     = {10.1016/j.geomphys.2010.05.014}
}

@article{ColomboGentiliSabadiniStruppa2010,
  author  = {Colombo, Fabrizio and Gentili, Graziano and Sabadini, Irene and Struppa, Daniele C.},
  title   = {Non commutative functional calculus: bounded operators},
  journal = {Complex Anal. Oper. Theory},
  volume  = {4},
  number  = {4},
  pages   = {821--843},
  year    = {2010},
  doi     = {10.1007/s11785-009-0015-3}
}

@article{Dou202000425,
	Title = {A representation formula for slice regular functions over slice-cones in several variables},
    Author  = {Xinyuan Dou and Guangbin Ren and Irene Sabadini},
	Year = {2023},
	JOURNAL = {Ann. Mat. Pura Appl.},
}

@article{Dou202000426,
	Title = {Extension theorem and representation
    formula in non-axially-symmetric domains for slice regular functions},
    Author = {Xinyuan Dou and Guangbin Ren and Irene Sabadini},
	JOURNAL = {J. Eur. Math. Soc.},
    volume={25},
    number={9},
    pages={ 3665–3694},
    Year = {2023},
}

@article{GhiloniStoppato2024,
  author  = {Ghiloni, Riccardo and Stoppato, Caterina},
  title   = {Quaternionic slice regularity beyond slice domains},
  journal = {Math. Z.},
  volume  = {306},
  number  = {3},
  pages   = {55},
  year    = {2024},
  doi     = {10.1007/s00209-024-03434-7}
}

@article{Perotti2024,
  author  = {Perotti, Alessandro},
  title   = {A Local {C}auchy Integral Formula for Slice-Regular Functions},
  journal = {Comput. Methods Funct. Theory},
  volume  = {24},
  pages   = {185--203},
  year    = {2024},
  doi     = {10.1007/s40315-023-00485-5}
}

@article{ColomboSabadiniStruppa2012,
  author  = {Colombo, Fabrizio and Sabadini, Irene and Struppa, Daniele C.},
  title   = {Algebraic properties of the module of slice regular functions in several quaternionic variables},
  journal = {Indiana Univ. Math. J.},
  volume  = {61},
  number  = {4},
  pages   = {1581--1602},
  year    = {2012},
  doi     = {10.1512/iumj.2012.61.4978}
}

@article{RenYang2020,
  author  = {Ren, Guangbin and Yang, Ting},
  title   = {Slice regular functions of several octonionic variables},
  journal = {Math. Methods Appl. Sci.},
  volume  = {43},
  number  = {9},
  pages   = {6031--6042},
  year    = {2020},
  doi     = {10.1002/mma.6344}
}

@article{GhiloniPerotti2022,
  author  = {Ghiloni, Riccardo and Perotti, Alessandro},
  title   = {Slice regular functions in several variables},
  journal = {Math. Z.},
  volume  = {302},
  number  = {1},
  pages   = {295--351},
  year    = {2022},
  doi     = {10.1007/s00209-022-03066-9}
}

@inproceedings{GhiloniPerotti2012Clifford,
  author    = {Ghiloni, Riccardo and Perotti, Alessandro},
  title     = {Slice regular functions of several {C}lifford variables},
  booktitle = {A{I}{P} {C}onference {P}roceedings},
  volume    = {1493},
  pages     = {734--738},
  year      = {2012},
  publisher = {American Institute of Physics},
  doi       = {10.1063/1.4765569}
}
\end{document}